\documentclass[11pt]{article}
\usepackage[a4paper,margin=1in]{geometry}
\usepackage{amsmath,amssymb,amsthm,mathtools}
\usepackage{mathpazo}            
\usepackage[T1]{fontenc}
\usepackage{microtype}
\usepackage{xcolor}
\usepackage{tikz}
\usetikzlibrary{patterns,arrows.meta,positioning}
\usepackage{booktabs}
\usepackage{enumitem}
\usepackage{graphicx}            
\usepackage[colorlinks=true,linkcolor=blue!50!black,citecolor=blue!50!black,urlcolor=blue!50!black]{hyperref}
\usepackage[nameinlink,noabbrev]{cleveref}

\usepackage[round]{natbib}

\setlist[itemize]{topsep=3pt,itemsep=2pt,parsep=0pt}
\setlist[enumerate]{topsep=3pt,itemsep=2pt,parsep=0pt}
\allowdisplaybreaks
\theoremstyle{plain}
\newtheorem{theorem}{Theorem}[section]
\newtheorem{lemma}[theorem]{Lemma}

\newtheorem{corollary}[theorem]{Corollary}

\newtheorem{conjecture}[theorem]{Conjecture}
\theoremstyle{definition}
\newtheorem{definition}[theorem]{Definition}

\newcommand{\C}{\mathbb C}
\newcommand{\R}{\mathbb R}

\newcommand{\Tr}{\operatorname{tr}}

\newcommand{\cA}{\mathcal A}

\newcommand{\cH}{\mathcal H}
\newcommand{\cP}{\mathcal P}

\newcommand{\op}{\mathrm{op}}
\newcommand{\F}{\mathrm F}

\newcommand{\norm}[1]{\left\lVert #1\right\rVert}
\newcommand{\abs}[1]{\left\lvert #1\right\rvert}
\newcommand{\KL}[2]{S\!\left(#1\,\middle\|\,#2\right)}

\newcommand{\nlsum}{\sum\nolimits}

\numberwithin{equation}{section}

\def\isarxiv{1}
\begin{document}

\title{\bfseries An Algebraic Matrix Spencer Theorem}
\ifdefined\isarxiv
\author{Emrullah Akbas\thanks{\texttt{emrullah.akbas@tum.de}. Math Department, TU Munich, Garching, Germany}  \and Suvrit Sra\thanks{\texttt{s.sra@tum.de}. Math Department, TU Munich, Garching, Germany} }
\else
\author{}
\fi
\date{}
\maketitle

\begin{abstract}
We develop an algebraic approach to matrix discrepancy based on the representation theory of finite-dimensional C$^*$-algebras. As an application, we resolve a substantial structured special case of the Matrix Spencer conjecture. In particular, we show that for every family of  contractions $A_1,\ldots,A_n$ that are  contained in a finite-dimensional $C^*$-algebra $\cA$ with $\dim_\C (\cA) \lesssim n$, there exists signs $x\in\{\pm1\}^n$ such that $\norm{\sum_{i=1}^n x_i A_i} \le O(\sqrt n)$. As a noteworthy special case, our main result also resolves the Group Spencer conjecture of~\citep{Bandeira24MatrixDiscrepancyBlog}. We furthermore prove that Matrix Spencer continues to  hold for low-rank perturbations of matrix families coming from an $C^*$-algebra of small dimension.

\end{abstract}

\tableofcontents

\section{Introduction}
Classical discrepancy theory asks for signings that balance a collection of vectors $a_{1},\dots, a_{n}\in\R^{d}$ satisfying $\| a_{i}\|_{\infty}\le 1$, for instance by finding a vector $x\in\{\pm1\}^{n}$ to minimize the discrepancy $\| \sum_{i=1}^{n}x_{i}a_{i}\|_{\infty}$. Choosing the signs  $(x_{1},\dots, x_{n})$  at random gives a coloring with discrepancy $O(\sqrt{n \log(n)})$ by a Chernoff/union bound argument. Remarkably, in a seminal result, \citet{Spencer1985} showed that one can do better and find a coloring with discrepancy $O(\sqrt n)$, improving on the random coloring by a logarithmic factor in the dimension. The Matrix Spencer conjecture~\citep{zouzias2012, meka2014} is a noncommutative analogue in which  vectors are replaced by symmetric matrices and the $\ell_\infty$ norm by the operator norm.
  
\begin{conjecture}[Matrix Spencer]
There exists a universal constant $C$, such that for any choice of $n$ symmetric matrices $(A_{i})_{i=1}^n \in \mathbb R^{n\times n}$ satisfying $\|A_i\|_{\op}\le 1$ there is a coloring $x\in\{\pm1\}^{n}$, such that
\begin{equation}
    \label{eq:1}
 \Bigl\|\sum\nolimits_{i=1}^n x_i A_i\Bigr\|_{\op} \le C \sqrt{n}.
\end{equation}
\end{conjecture}
 The non-commutative Khintchine inequality (or a matrix concentration inequality) shows that, as in the vector setting, randomly choosing the signs $x$ yields a coloring with discrepancy  $\left\|\sum_{i=1}^n x_i A_i\right\|_{\rm op}=O(\sqrt{ n \log n})$. Thus, Matrix Spencer asks whether this logarithmic factor can always be removed by a more careful choice of signs.  For diagonal, hence commuting, matrices the problem reduces to the classical discrepancy question, and Spencer's theorem gives the desired $O(\sqrt n)$ bound.  In contrast, for highly noncommutative families, sharper versions of the non-commutative Khintchine inequality \citep{BandeiraBoedihardjoVanHandel2023} suggest that random signs can already be close to optimal in many regimes.  In fact, the noncommutative Khintchine inequalities motivate the following variance-sensitive Matrix Spencer conjecture, which also appears as Remark 4.25 in the problem collection shared by~\citet{bandeira2016}. 
 \begin{conjecture}[Variance-sensitive Matrix Spencer]
 \label{conj:var}
 There exists a universal constant $C$, such that for any choice of $n$ symmetric matrices $A_{1}, \dots, A_{n} \in \mathbb R^{n\times n}$ satisfying $\|A_i\|_{\op}\le 1$ there exists a coloring $x\in\{\pm1\}^{n}$, such that
\begin{equation}
\label{eq:2}
\Bigl\|\nlsum_{i=1}^n x_i A_i\Bigr\|_{\op} \le C \Bigl\|\nlsum_{i=1}^n A_i^2\Bigr\|_{\op}^{1/2}.
\end{equation}
\end{conjecture}
Conjecture~\ref{conj:var} is a strengthening of  Matrix Spencer conjecture, which can be obtained by upper bounding the variance $\left\|\sum_{i=1}^n A_i^2\right\|_{\op}^{1/2}$ by $\sqrt{n}$. 
For special cases such as  rank-1 matrices, the validity of this conjecture has been shown~\citep{KyngLuhSong2020FourDeviations}. However, we will show here that this conjecture cannot hold in full generality for all matrix families. Our lower bound construction is given by a family of diagonal matrices, we expect that this conjecture may hold if matrices are sufficiently non-commutative.

\subsection{Summary of Results}
We resolve the Matrix Spencer conjecture in the case that $A_1, \dots, A_n$ is a collection of contractions belonging to a finite-dimensional $C^*$-algebra $\cA$ with $\dim_\C (\cA) \le O( n)$. 
\begin{theorem}[Algebraic Matrix Spencer] \label{thm:algebraic:spencer}
    Let $\cA$ be a finite-dimensional $C^*$-algebra of dimension $\dim_\C (\cA) \le O( n)$ and $A_1, \dots, A_n\in \cA$ have $\norm{A_i}\le 1$. Then there exists $x\in\{\pm1\}^{n}$ such that $\norm{\sum_{i=1}^n x_i A_i} \le O(\sqrt n)$.
\end{theorem}

Our result captures regimes in which the ambient dimension is not the right  measure of complexity for discrepancy. Standard Matrix Spencer-type bounds depend on the dimension of the matrices and therefore give a looser estimate. In contrast, our theorem depends only on the dimension of the $C^*$-algebra generated by the matrices. Thus, whenever the matrices generate a low-dimensional algebra inside a much larger matrix algebra, our bound captures a novel improvement over the standard Matrix Spencer-type bound.


As an important special case, our main result also implies the Group Spencer conjecture, which was posed in a wonderful blog post by ~\citet{Bandeira24MatrixDiscrepancyBlog}.
\begin{corollary}[Group Spencer]
    Let $G$ be a finite group of order $n$. There exist signs $x_g \in\{\pm1\} $ for each $g\in G$ such that for any unitary representation $\rho:G\to \text{GL}(n)$ we have $\|\sum_{g\in G} x_g \rho(g)\| \le O(\sqrt n)$.
\end{corollary}
\begin{proof}
    Let $\rho:G\to \text{GL}(n)$ denote the left regular representation of $G$. Consider the $C^*$-algebra $\cA$ generated by $\{\rho(g): g\in G\}$. By Peter--Weyl there exists a unitary $U$ such that,
    \begin{equation}
    \label{eq:weyl}
    \rho(g) = U (\bigoplus_{\pi\in\hat G} I_{d_\pi}\otimes \pi) U^*,
    \end{equation}
   where $\hat G$ denotes  the set of isomorphism classes of irreducible representations of $G$ and $d_\pi$ denote the degree of $\pi\in\hat{G}$.
    Consequently, we also get $\cA \cong \bigoplus_{\pi\in\hat{G}}M_{d_\pi}(\C)$ and therefore also 
    $\dim_\C (\cA) = \sum_{\pi\in\hat{G}}d_\pi^2=|G|$. By Theorem \ref{thm:algebraic:spencer}, applied to the family $\{\rho(g)\:g\in G\}$, there exists signs $x_g\in\{\pm 1\}$ such that 
    \begin{equation}
    \label{eq:3}
      \Bigl\| \sum_{g\in G} x_g \rho(g)\Bigr\| \le O(\sqrt{|G|}).  
    \end{equation}
    Using the block-decomposition~\eqref{eq:weyl} we have the identification
    \begin{equation}
    \label{eq:4}
        \sum_{g\in G} x_g \rho(g) \cong \bigoplus_{\pi\in\hat G} I_{d_\pi}\otimes ( \sum_{g\in G} x_g \pi(g) ).
    \end{equation}
    Now take operator norms in~\eqref{eq:3} and use inequality~\eqref{eq:3} to obtain
    \begin{equation}
    \label{eq:5}
    \max_{\pi\in\hat{G}} \Bigl\| \sum_{g\in G} x_g \pi(g) \Bigr\| =  \Bigl\| \sum_{g\in G} x_g \rho(g)\Bigr\| \le O(\sqrt{|G|}).
    \end{equation}
    Now, let $\rho:G\to \text{GL}(n)$ be any finite-dimensional unitary representation of $G$, and use complete reducibility to obtain the  decomposition $\rho \cong \bigoplus_{\pi\in\hat G} I_{m_\pi}\otimes \pi$. With this decomposition repeat the above steps from~\eqref{eq:3}--\eqref{eq:5} to finally obtain
    $\bigl\| \sum_{g\in G} x_g \rho(g)\bigr\| \le O(\sqrt{|G|})$. 
\end{proof}

\subsection{Extension to a bigger class of matrices}
A further (minor) contribution of this work is to combine the algebraic method  with the existing rank-constrained Matrix Spencer results, giving new families of possibly full-rank matrices for which the Matrix Spencer conjecture holds. In particular, we prove a hybrid theorem for decompositions
\[
A_i=B_i+L_i,
\]
where the \(B_i\)'s lie in a finite-dimensional \(C^*\)-algebra of dimension \(O(n)\), while the \(L_i\)'s satisfy the  Frobenius/rank hypotheses of Bansal--Jiang--Meka \cite{BansalJiangMeka2024MatrixSpencer}.

\begin{theorem}[Matrix Spencer for rank-constrained perturbations of algebraic families]
\label{thm:low-rank-plus-algebra}
For every \(C_0\ge 1\), there is a constant \(C=C(C_0)\) with the following property. Let
\[
A_i=B_i+L_i,
\qquad
i=1,\ldots,n,
\]
be self-adjoint matrices. Assume that \(B_i\in\mathcal A\) for a finite-dimensional \(C^*\)-algebra \(\mathcal A\) satisfying $\dim_{\mathbb C}\mathcal A\le C_0n$, and assume $\|B_i\|\le 1$, $\|L_i\|\le 1$, and $\|L_i\|_F^2\le \frac{n}{\log^3(en)}$ for every \(i\in[n]\). Then, there exist signs \(x_1,\ldots,x_n\in\{\pm1\}\) such that
\[
\Bigl\|
\nlsum_{i=1}^n x_iA_i
\Bigr\|
\le
C(C_0)\sqrt n.
\]
\end{theorem}

\subsection{Discussion and Related Work}
A substantial line of work has treated various cases of the  Matrix Spencer conjecture.  \cite{hopkins2022matrix} proved the conjectured $O(\sqrt n)$ bound for matrices of rank at most $n^{1/4}$, establishing a novel connection between matrix discrepancy and    quantum communication complexity. \cite{dadush2022matrix} developed a partial-coloring framework via mirror descent, obtaining low-rank and block-diagonal matrix discrepancy bounds.  \cite{BansalJiangMeka2024MatrixSpencer} later proved Matrix Spencer conjecture for matrices of rank at most $n/\log^3 n$. Their proof exploits the close connection of matrix discrepancy and random matrix theory; in particular it employs the sharp noncommutative Khintchine inequalities of \cite{BandeiraBoedihardjoVanHandel2023} within a very clever decomposition argument.

This paper studies a complementary source of structure. Instead of assuming low rank for the individual matrices, we assume that the matrices lie in a finite-dimensional \(C^*\)-algebra of low intrinsic dimension. The relevant dimension for discrepancy is not the ambient Hilbert-space dimension. Indeed, by  Wedderburn,  $ \cA  \cong 
\bigoplus_{\alpha\in\Lambda} I_{r_\alpha}\otimes M_{d_\alpha}(\mathbb C)$ and any element in the algebra has the form \(A_i=\bigoplus_\alpha I_{r_\alpha}\otimes A_i^{(\alpha)}\), so that every signed sum then satisfies
\[
\bigl\|\nlsum_i x_iA_i\bigr\|
=
\max_{\alpha}\ 
\bigl\|\nlsum_i x_iA_i^{(\alpha)}\bigr\|.
\]
Thus,  one can observe that the multiplicities \(r_\alpha\) do not have an effect to the operator norm, and the discrepancy is governed by the intrinsic block dimensions \(d_\alpha\), not by the ambient dimension of the Hilbert space.

\vskip4pt
\noindent\textbf{Concurrent Work.}
In a concurrent and independent work~\citet{BandeiraBolcskei2026MatrixDiscrepancy} also resolved the Group Spencer conjecture as in Corollary 1.4 using different techniques based on random matrix theory.
\section{Preliminaries}
\label{sec:preliminaries}
All Hilbert spaces in this work are finite-dimensional.  We write $\norm{\cdot}$ for the
operator norm and $\norm{\cdot}_{\F}$ for the Frobenius norm.  For
self-adjoint $M$, the norm is dual to trace norm:
\[
  \norm{M}=\sup\{\abs{\Tr(MX)}:\norm{X}_1\le1,\ X=X^*\}.
\]
For density matrices (i.e., positive semidefinite and unit trace) $X,Y$, the quantum relative entropy is defined as
\begin{equation}
\label{eq:kl}
  \KL{X}{Y}:=\Tr\bigl(X(\log X-\log Y)\bigr).
\end{equation}

\subsection{Partial Coloring}
We recall the standard notion of a partial coloring. 

\begin{definition}[Partial coloring]
A partial coloring of the matrices $A_1, \dots, A_n$ with discrepancy $\Delta>0$ is a vector $x\in[-1,1]^n$ such that $|x_i| = 1$ for a constant fraction of its coordinates $i\in[n]$, and $\|\sum_{i=1}^nx_i A_i\|\le \Delta$.
\end{definition}
For a collection of matrices $A_1, \dots, A_n$ we denote by
\begin{equation}
\label{eq:cvxbdy}
K  = \{ x\in \mathbb R^{n}: \big\|\nlsum_{i=1}^n x_i A_i\bigr\| \leq \Delta  \},
\end{equation}
the convex body of partial fractional colorings with discrepancy at most $\Delta$. To show the existence of \emph{partial coloring}  in $K$  with a constant fraction of its coordinates being integral, it suffices to show that $K$ has sufficiently large Gaussian volume $\gamma(K)\geq \exp(-\Omega(n))$ (\cite{Gluskin1989Extremal, Giannopoulos1997VectorBalancing}). We use a variant of the Gaussian partial-coloring theorem due to Rothvoss~\cite[Lemma~9]{Rothvoss2017Constructive}

\begin{theorem}[Gaussian partial coloring]
\label{thm:partial-coloring}
There are absolute constants $a$,$c$, and $C>0$ with the following property. Let $K\subset\R^s$ be convex body.  If $\gamma_s(K)\ge 2^{-as}$, then for every point $y\in(-1,1)^s$ there is an increment
$x\in C K\cap[-1,1]^s$ such that $x+y\in[-1,1]^s$ and at least
$c\cdot s$ coordinates of $x+y$ lie in $\{\pm1\}$.
\end{theorem}

Moreover, we will also need the version of the Gaussian partial coloring theorem relative to a subspace, which is also due to \cite{Rothvoss2017Constructive}.

\begin{theorem}[Gaussian partial coloring relative to a subspace]
\label{thm:partial-coloring-relative}
There are absolute constants \(c,\eta,C>0\) with the following property. Let
\(H\subseteq\mathbb R^s\) be a linear subspace satisfying $\dim H\ge (1-\eta)s$, and let \(K\subseteq H\) be a symmetric convex set satisfying $\gamma_H(K)\ge \exp(-\eta s)$.
Then for every \(y\in(-1,1)^s\), there exists an increment \(x\in CK\) such that $y+x\in[-1,1]^s$ and $|\{i\in[s]: |y_i+x_i|=1\}|\ge cs.$
\end{theorem}

We further recall in this section tools we use from \cite{dadush2022matrix}.
In their work Dadush, Jiang, and Reis recast the problem of proving a Gaussian-measure lower bound of the discrepancy body as a covering-number estimate for the polar body.  We use it only up to absolute rescaling constants.

\begin{theorem}[Covering-to-measure principle, \cite{dadush2022matrix} Lemma 3.3]
\label{thm:covering-measure}
Let $K\subset\R^s$ be symmetric and convex.  The following are equivalent,
\begin{enumerate}
    \item[\textup{(1)}] $N\!\left(K^\circ,\frac1sB_\infty^s\right)\le 2^{O(s)}$,
    \item[\textup{(2)}] $N\left(K^\circ,\frac{1}{\sqrt s}B_2^s\right)\le 2^{O(s)}$,
    \item[\textup{(3)}] $\gamma_s(K)\ge 2^{-O(s)}$.
\end{enumerate}
Moreover, we will use the following parameterized form of the implication \(\textup{(2)}\Rightarrow\textup{(3)}\). If \(K\subseteq\mathbb R^s\) is a symmetric convex body, and for some $r\ge s$, $N\left(K^\circ,\frac{1}{\sqrt s}B_2^s\right)\le 2^{Cr}$, then $\gamma_s(cK)\ge 2^{-C'r}$, where $c,C'>0$ depend only on $C$.
\end{theorem}

Another tool we will need  is the following relative-entropy net mechanism of  ~\cite[Lemma~3.6]{dadush2022matrix}.

\begin{theorem}[Relative-entropy net to covering]
\label{thm:renet-covering}
Let $\Delta$ be a block-diagonal spectraplex, and for $S\subseteq[n]$ define
\[
  \mathcal A_S(X)=\bigl(\Tr(A_iX)\bigr)_{i\in S}\in\R^S.
\]
Suppose $T_0\subset\Delta$ has size $\exp(O(s))$, where $s=|S|$, and every
$X\in\Delta$ has some $Y\in T_0$ with $\KL{X}{Y}\le D$.  Then
\[
  N\!\left(\mathcal A_S(\Delta),
    C\sqrt{D/s}\,B_\infty^S\right)
  \le \exp(O(s)).
\]
\end{theorem}

Finally, the  following equal-block relative-entropy net we consider here is the block-diagonal spectraplex net of
~\cite[Theorem~3.9]{dadush2022matrix}, whose construction uses
operator-norm covering estimates for Schatten classes.

\begin{theorem}[Equal-block relative-entropy net]
\label{thm:equal-block-entropy-net}
Let \(\Delta_{h,m}\) be the block-diagonal spectraplex on total dimension \(m\), with all blocks of size \(h\). For every parameter \(q\ge 1\), there exists a relative-entropy net \(T\subseteq \Delta_{h,m}\) such that
\[
|T|\le \exp(Cq)
\]
and for every \(X\in\Delta_{h,m}\) there is \(Y\in T\) satisfying
\[
S(X||Y)
\le
C\max\left\{1,\log\left(\frac{2hm}{q}\right)\right\}.
\]
\end{theorem}

\section{Proof of the Main Theorem}
\subsection{Block-Diagonal case via a Multiscale Block-Partial Coloring Theorem}
In this section we will study the matrix discrepancy of  block diagonal matrices. This setting has been treated previously by  \cite{dadush2022matrix}, who proved a Matrix Spencer bound for block diagonal matrices with a fixed block-size parameter $h$, obtaining a discrepancy bound of order $\sqrt{n\max{(1,\log(hm/n))}}$. Our argument builds on the relative-entropy net and partial-coloring framework of Dadush, Jiang, and Reis. While their argument treats the block structure at a single scale, through a relative-entropy net for  the block-diagonal spectraplex, our
contribution is a multiscale refinement of this argument for heterogeneous block systems:
We decompose the common block structure into groups of comparable block sizes and assign separate entropy budgets to the different scales. This replaces the homogeneous complexity parameter $hm$ by the summable quantity
\[
Q=\sum_j h_jm_j,
\]
where $h_j$ is the largest block size in the $j$-th group and $m_j$ is its total dimension.
More precisely, we consider the following multiscale decomposition for block diagonal matrices.

\begin{definition}[Block stratification] \label{def:stratification}
Let
\[
H=\bigoplus_{\alpha\in\mathcal B}H_\alpha
\]
be a fixed common block decomposition, and write $b_\alpha:=\dim H_\alpha .$ A \emph{block stratification} is a partition
\[
\mathcal B=\mathcal P_1\sqcup\cdots\sqcup \mathcal P_J
\]
of the block index set into nonempty groups, called strata. For a stratum $\mathcal P_j$, define
\[
\cH_{\mathcal P_j}:=\bigoplus_{\alpha\in\mathcal P_j}\cH_\alpha,
\qquad
h_j:=\max_{\alpha\in\mathcal P_j} b_\alpha,
\qquad
m_j:=\sum_{\alpha\in\mathcal P_j} b_\alpha,
\]
and $q_j:=h_jm_j$. The total stratified block complexity is $Q(\mathcal P):=\sum_{j=1}^J q_j= \sum_{j=1}^J h_jm_j$.

We also write
\[
\cH_{\max}:=\max_{\alpha\in\mathcal B}b_\alpha,
\qquad
M:=\sum_{\alpha\in\mathcal B}b_\alpha,
\]
for the largest block size and total ambient dimension of the compressed block system.
\end{definition}

\begin{theorem}[Block-Diagonal Matrix Spencer]\label{thm:block-diag-spencer}
    Let $A_1,\dots,A_n$ be self-adjoint satisfying $\|A_i\|\le1$ and which are block diagonal with respect to a fixed common block decomposition $H=\bigoplus_{\alpha\in\mathcal B}H_\alpha$. Consider  block stratification as in Definition \ref{def:stratification} and let $(h_j,m_j,q_j,Q)$ be the associated parameters.
    Assume that these satisfy $Q = \sum_{j=1}^J h_j m_j \le O(n)$ and $J \le O(\log(e n ))$. Then, there exists a coloring $x\in\{\pm 1\}^n$ such that
    \[
    \Bigl\| \nlsum_{j=1}^{n} x_i A_i\Bigr\| \le O(\sqrt{n}).
    \]
\end{theorem}

\subsection{Partial Coloring from Entropy Nets}
In this section we prove the following  partial coloring lemma and then we use it to deduce Theorem \ref{thm:block-diag-spencer}. 

\begin{lemma}[Multiscale block partial coloring]
\label{lem:stagewise-multiscale-block-pc}
 Let $A_1,\dots,A_n$ be self-adjoint satisfying $\|A_i\|\le1$ and which are block diagonal with respect to a fixed common block decomposition $H=\bigoplus_{\alpha\in\mathcal B}H_\alpha$. Consider  block stratification as in Definition \ref{def:stratification} and let $(h_j,m_j,q_j,Q)$ be the associated parameters.
Then, for every active set \(S\subseteq[n]\), with \(s:=|S|\), and every shift $y\in(-1,1)^S$,
there exists an increment $x\in[-1,1]^S$
such that \(x+y\in[-1,1]^S\) and at least \(\Omega(s)\) coordinates of \(x+y\) lie in \(\{\pm1\}\), and
\[
\left\|
\sum_{i\in S}x_iA_i
\right\|
\le \Delta(s),
\]
where
\[
\Delta(s)
\le
\begin{cases}
O(\sqrt{s\log\!\left(e+\dfrac{Q}{s}\right)}), & s\ge J,\\[1.2em]
O(\sqrt{s\log\!\left(e+\dfrac{b_{\max}M}{s}\right)}), & s<J.
\end{cases}
\]
Moreover, if \(Q\le O(n)\), then
\[
b_{\max}\le O(\sqrt{n}),
\qquad
M\le O(n),
\qquad
b_{\max}M\le O(n^{3/2}).
\]
\end{lemma}

The proof of Lemma~\ref{lem:stagewise-multiscale-block-pc} follows the standard convex-geometric partial-coloring strategy. Fix an active set \(S\subseteq[n]\), write \(s=|S|\), and consider the discrepancy body
\[
K_S(t):=
\left\{
z\in\mathbb R^S:
\left \|
\sum_{i\in S}z_iA_i
\right\|
\le t
\right\}.
\]
By Rothvoss's Gaussian partial-coloring theorem \ref{thm:partial-coloring} in the form used by \cite[Theorem~3.1]{dadush2022matrix} shows that it is enough to prove that \(K_S(t)\) has Gaussian measure at least \(\exp(-O(s))\). We obtain this measure lower bound as in \cite{dadush2022matrix} through polar covering estimates. By trace duality, the polar of the operator-norm discrepancy body is controlled by the image of the block-diagonal spectraplex under the map
\[
X\mapsto \bigl(\operatorname{tr}(A_iX)\bigr)_{i\in S}.
\]
Thus the task reduces to constructing small relative-entropy nets for the relevant block-diagonal spectraplex.

We use a multiscale decomposition to obtain much finer construction of relative-entropy nets. If one treats all blocks at the largest scale, the entropy loss is governed by the crude parameter \(b_{\max}M\). Instead, we first build a variable-budget net for a single stratum of blocks. We then combine these one-stratum nets using a product construction over the block stratification. A discretization of the stratum masses, together with Jensen's inequality, converts the individual costs into the sharper global parameter $Q=\sum_j h_jm_j.$ This yields an entropy error of order $\log\left(e+\frac{Q}{s}\right)$ for \(s\ge J\). Finally, the relative-entropy net the yields  a Gaussian-measure lower bound and hence the existence of a partial-coloring.

We now implement this strategy in three steps. First, we construct relative-entropy nets for the spectraplex associated with a single stratum. Second, we combine these nets across strata by a product construction, choosing the entropy budget of each stratum according to its mass. Third, we use the resulting relative-entropy net to obtain a Gaussian-measure lower bound for the discrepancy body and hence a partial coloring.

\paragraph{Step 1: One-stratum relative-entropy net.}
We begin with the local net construction for a single stratum of the block decomposition. The point of this lemma is to isolate the one-scale estimate: if a collection of blocks has total dimension \(m\) and largest block size \(h\), then its block-diagonal spectraplex admits relative-entropy nets with a tunable size-error tradeoff. Namely, for every parameter \(\tau>0\), one obtains a net of size at most \(\exp(C\tau)\), while the relative-entropy error is logarithmic in \(hm/\tau\). This one-stratum estimate will later be combined across the different strata by assigning a separate value of \(\tau\) to each stratum.

\begin{lemma}[One-stratum relative-entropy net]
\label{lem:one-stratum-net}
Let $H_P=\bigoplus_{\alpha\in P}H_\alpha$
be a finite-dimensional Hilbert space equipped with a fixed block decomposition. Write
\[
b_\alpha:=\dim H_\alpha,
\qquad
m:=\sum_{\alpha\in P}b_\alpha,
\qquad
h:=\max_{\alpha\in P}b_\alpha.
\]
Let
\[
\Delta_P
:=
\left\{
X\succeq 0:
\operatorname{tr}X=1,\
X=\bigoplus_{\alpha\in P}X_\alpha
\text{ with }X_\alpha\in B(H_\alpha)
\right\}
\]
be the block-diagonal spectraplex associated with this decomposition. Then for every \(\tau>0\), there exists a finite relative-entropy net $T_P(\tau)\subseteq \Delta_P$ of the spectraplex $\Delta_P$ with $|T_P(\tau)|\le \exp(C\tau)$ such that  for every \(X\in\Delta_P\) there is \(Y\in T_P(\tau)\) satisfying
\[
S(X||Y)
\le
C\log\left(e+\frac{hm}{\tau}\right).
\]
Here \(C>0\) is a universal constant.
\end{lemma}

\begin{proof}
For \(0<\tau<1\), we let $T_P(\tau):=\left\{\frac{I_P}{m}\right\}$, where \(I_P\) denotes the identity on \(H_P\), so that \(|T_P(\tau)|=1\le \exp(C\tau)\). Moreover, for every \(X\in\Delta_P\), using that \(h\ge1\) and \(\tau<1\), we have
\[
S\left(X\middle\|\frac{I_P}{m}\right)
=
\operatorname{tr}(X\log X)+\log m
\le
\log m \le 
C\log\left(e+\frac{hm}{\tau}\right),
\]
which proves the claim in the case \(0<\tau<1\).

Now, for the following assume \(\tau\ge1\). Since every block in \(P\) has dimension at most \(h\), we may pad and group the blocks, if necessary, into equal blocks of size \(H\le 2h\) and denote with \(M\) denote the resulting total dimension, satisfying $M\le 4m.$
Moreover, the original block-diagonal spectraplex \(\Delta_P\) embeds into the corresponding equal-block spectraplex with block size \(H\) and total dimension \(M\).
Now, by Theorem \ref{thm:equal-block-entropy-net}, the equal-block relative-entropy net theorem, applied with block size \(H\), total dimension \(M\), and parameter \(q :=\left\lceil \min\{\tau,hm\}\right\rceil\), there exists a relative-entropy net of size $\exp(O(q))\le \exp(O(\tau)).$
Since $HM\le 8hm$, 
the relative-entropy error is at most $O\left(\max\left\{1,\log\left(\frac{HM}{q}\right)\right\}\right)
\le
C\log\left(e+\frac{hm}{\tau}\right).$
Restricting the resulting net back to the original block-diagonal spectraplex,  gives a set \(T_P(\tau)\subseteq\Delta_P\) satisfying $|T_P(\tau)|\le \exp(C\tau)$
and such that every \(X\in\Delta_P\) has some \(Y\in T_P(\tau)\) with
\[
S(X\|Y)
\le
C\log\left(e+\frac{hm}{\tau}\right),
\]
which proves our lemma.
\end{proof}

\paragraph{Step 2: Product relative-entropy net over strata.}
We now combine the one-stratum nets across the block stratification. The main point is that a global density matrix need not distribute its trace uniformly across the strata. We therefore first discretize the vector of stratum masses and then, conditional on a choice of masses, apply Lemma~\ref{lem:one-stratum-net} separately inside each stratum with a mass-dependent net-size parameter. The product of these one-stratum nets has size \(\exp(O(s))\), while Jensen's inequality converts the sum of the stratum-wise entropy errors into the single logarithmic term $\log\left(e+\frac{Q}{s}\right)$ where $Q=\sum_{j=1}^J h_jm_j$. This is the step where the multiscale decomposition improves the crude one-scale parameter \(b_{\max}M\) over the parameter \(Q\).

\begin{lemma}[Product relative-entropy net over strata]
\label{lem:product-strata-net}
Let $H=\bigoplus_{\alpha\in\mathcal B}H_\alpha$
be a fixed common block decomposition, and consider  block stratification as in Definition \ref{def:stratification} and let $(h_j,m_j,q_j,Q)$ be the associated parameters.
Let
\[
\Delta
:=
\left\{
X\succeq 0:
\operatorname{tr}X=1,\ 
X=\bigoplus_{\alpha\in\mathcal B}X_\alpha
\text{ with }X_\alpha\in B(H_\alpha)
\right\}
\]
be the full block-diagonal spectraplex and \(S\subseteq[n]\) with \(s:=|S|\), and \(s\ge J\). Then there is a relative-entropy net $T_S\subseteq \Delta$ of the full block-diagonal spectraplex with $|T_S|\le \exp(Cs)$ such that 
 for every \(X\in\Delta\) there exists \(Y\in T_S\) satisfying
\[
S(X\|Y)
\le
C\log\left(e+\frac{Q}{s}\right),
\]
where  \(C>0\) is some  universal constant
\end{lemma}
\begin{proof}
For each stratum \(\mathcal P_j\), let
\[
H_j:=\bigoplus_{\alpha\in\mathcal P_j}H_\alpha
\]
and let \(\Pi_j\) denote the orthogonal projection onto \(H_j\). Let
\[
\Delta_j
:=
\left\{
X\succeq0:
\operatorname{tr}X=1,\ 
X \text{ is block diagonal on }H_j
\right\}
\]
be the stratum spectraplex. Fix \(X\in\Delta\) and define its mass on the \(j\)-th stratum by $\alpha_j:=\operatorname{tr}(\Pi_jX\Pi_j)$
Then we have that \(\alpha_j\ge0\) and $\sum_{j=1}^J\alpha_j=1.$
Moreover, for \(\alpha_j>0\) we define the normalized state $$X_j:=\alpha_j^{-1}\Pi_jX\Pi_j\in\Delta_j.$$ Now, we discretize the vector of stratum masses. To this end, let
\[
\Gamma_s
:=
\left\{
\beta\in\mathbb R_+^J:
\beta_j=\frac{k_j+1}{s+J},\ k_j\in\mathbb Z_{\ge0},\
\sum_{j=1}^J k_j=s
\right\}.
\]
Using \(s\ge J\), we have $|\Gamma_s|
=
\binom{s+J-1}{J-1}
\le
\binom{2s}{s}
\le
4^s.$
For every probability vector \(\alpha=(\alpha_1,\ldots,\alpha_J)\), we can choose integers
\(k_j\ge \lfloor s\alpha_j\rfloor\) with \(\sum_j k_j=s\) such that the corresponding
\(\beta\in\Gamma_s\) satisfies
\[
\beta_j
=
\frac{k_j+1}{s+J}
\ge
\frac{s\alpha_j}{s+J}
\ge
\frac{\alpha_j}{2}.
\]

Now fix some \(\beta\in\Gamma_s\) and  apply Lemma~\ref{lem:one-stratum-net} for each stratum \(j\) with $\tau_j:=s\beta_j$, which gives a net \(T_j(\tau_j)\subseteq\Delta_j\) satisfying $|T_j(\tau_j)|\le \exp(C\tau_j)$ and relative-entropy error $C\log\left(e+\frac{h_jm_j}{s\beta_j}\right).$
Define the product net associated with \(\beta\) by
\[
T(\beta)
:=
\left\{
\bigoplus_{j=1}^J \beta_jY_j:
Y_j\in T_j(s\beta_j)
\right\}.
\]
Indeed, each element of \(T(\beta)\) belongs to \(\Delta\), since each \(Y_j\) has trace one and
\(\sum_j\beta_j=1\).
Finally we let $T_S:=\bigcup_{\beta\in\Gamma_s}T(\beta)$, which then is of size at most $$|T_S|
\le |\Gamma_s| \max_{\beta\in\Gamma_s}|T(\beta)|
\le
|\Gamma_s| \max_{\beta\in\Gamma_s} \prod_{j=1}^J \exp(Cs\beta_j) \le
|\Gamma_s|\exp(Cs)
\le
\exp(C's)$$ for some constant $C'>0$. It then remains to prove the relative-entropy approximation guarantee. Fix \(X\in\Delta\), and choose
\(\beta\in\Gamma_s\) such that $\beta_j\ge \frac{\alpha_j}{2}$  for all  $ j$.
For each \(j\) with \(\alpha_j>0\), choose \(Y_j\in T_j(s\beta_j)\) such that
\[
S(X_j\|Y_j)
\le
C\log\left(e+\frac{h_jm_j}{s\beta_j}\right),
\]
while for indices \(j\) with \(\alpha_j=0\) we choose \(Y_j\in T_j(s\beta_j)\) arbitrarily. Let
\[
Y:=\bigoplus_{j=1}^J\beta_jY_j\in T(\beta)\subseteq T_S. 
\]
Using the block decomposition of relative entropy, we get
\begin{equation} \label{lemma35:eq1}
    S(X\|Y)
=
\sum_{\alpha_j>0}\alpha_j\log\frac{\alpha_j}{\beta_j}
+
\sum_{\alpha_j>0}\alpha_jS(X_j\|Y_j).
\end{equation}
Now, using \(\beta_j\ge\alpha_j/2\) and the concavity of the log, we can bound the relative entropy in \eqref{lemma35:eq1} as follows
\begin{eqnarray*}
   S(X\|Y)&\le&   \log2 + C\sum_{\alpha_j>0}\alpha_j
\log\left(e+\frac{h_jm_j}{s\beta_j}\right) \le C'\sum_{\alpha_j>0}\alpha_j
\log\left(e+\frac{2h_jm_j}{s\alpha_j}\right).\\
&\le&  C' \log\left(
\sum_{\alpha_j>0}
\alpha_j
\left(e+\frac{2h_jm_j}{s\alpha_j}\right)
\right)\le C' \log\left(e+\frac{2}{s}\sum_{j=1}^Jh_jm_j\right)
=
C' \log\left(e+\frac{2Q}{s}\right).
\end{eqnarray*}
This proves the desired relative-entropy estimate and completes the proof.
\end{proof}

\paragraph{Step 3: Partial coloring from relative-entropy nets.}
We now pass from relative-entropy nets to partial colorings. The preceding results construct small nets for the block-diagonal spectraplex. To use them for discrepancy, we apply the linear map
\[
A_S(X):=\bigl(\operatorname{tr}(A_iX)\bigr)_{i\in S}.
\]
By trace duality, the polar of the active discrepancy body is controlled by the image of the block-diagonal spectraplex under \(A_S\). By Theorem~\ref{thm:renet-covering}, a relative-entropy net for the spectraplex therefore gives a covering estimate for the polar discrepancy body. The polar covering estimate implies a Gaussian-measure lower bound for a suitable scaling of the discrepancy body, and Rothvoss's partial-coloring theorem then gives the desired partial coloring. We summarize this in the following lemma. It is the step that converts an entropy error \(D\) into a partial-coloring cost of order \(\sqrt{sD}\).

\begin{lemma}[Relative-entropy nets give partial colorings]
\label{lem:entropy-net-to-partial-coloring}
Let \(A_1,\ldots,A_n\) be self-adjoint satisfying $\|A_i\|\le1$ and which are block diagonal with respect to a fixed common block decomposition, and let \(\Delta\) denote the corresponding block-diagonal spectraplex. Let \(S\subseteq[n]\), \(s:=|S|\), and define
\[
A_S(X):=\bigl(\operatorname{tr}(A_iX)\bigr)_{i\in S}.
\]

Assume that there exists a finite relative-entropy net \(T_0\subseteq\Delta\) of the  block-diagonal spectraplex with $|T_0|\le \exp(C_0s)$ such that 
 for every \(X\in\Delta\) there exists \(Y\in T_0\) satisfying
\[
S(X\|Y)\le D
\]
Then for every point \(y\in(-1,1)^S\), there exists an increment \(z\in\mathbb R^S\) such that
\[
y+z\in[-1,1]^S, \quad |\{i\in S: |y_i+z_i|=1\} |\ge cs,
\]
and
\[
\left\|
\sum_{i\in S}z_iA_i
\right\|
\le
C\sqrt{sD}.
\]
Here \(c>0\) is an absolute constant, and \(C>0\) depends only on \(C_0\).
\end{lemma}

\begin{proof}
Let
\[
K_S
:=
\left\{
z\in\mathbb R^S:
\left\|
\sum_{i\in S}z_iA_i
\right\|
\le 1
\right\}
\]
be the discrepancy body of the active coordinates. We will show that a suitable scaling of \(K_S\) has Gaussian measure at least \(\exp(-O(s))\), and then apply the Gaussian partial-coloring theorem.

Using Theorem~\ref{thm:renet-covering} we  get that

\begin{equation}\label{lemma36:eq1}
    N\left(
A_S(\Delta),
C\sqrt{\frac{D}{s}}\,B_\infty^S
\right)
\le
\exp(C's),
\end{equation}
where \(B_\infty^S\) denotes the unit ball of \(\ell_\infty^S\). Set $\rho:=C\sqrt{\frac{D}{s}}$ and let
\[
\Delta_{\le 1}
:=
\left\{
X\succeq0:
\operatorname{tr}X\le1,\ 
X \text{ is block diagonal}
\right\}.
\]
Since \(A_i\) are contractions, we have $A_S(\Delta)\subseteq B_\infty^S.$
By discretizing the scalar trace parameter in \([0,1]\), the preceding covering estimate \eqref{lemma36:eq1} also gives
\begin{equation}\label{lemma36:eq2}
    N\left(
A_S(\Delta_{\le1}),
C\rho\,B_\infty^S
\right)
\le
\exp(C''s).
\end{equation}
Indeed, every \(X\in\Delta_{\le1}\) can be written as \(X=\lambda X_0\), with
\(\lambda\in[0,1]\) and \(X_0\in\Delta\) if \(\lambda>0\), and the scalar parameter \(\lambda\) may be discretized at mesh size comparable to \(\rho\). This increases the covering number only by an \(\exp(O(s))\) factor.
We now relate this covering estimate \eqref{lemma36:eq2} to the polar of \(K_S\). By trace duality,
\[
K_S^\circ
=
\left\{
\left(\operatorname{tr}(A_iW)\right)_{i\in S}:
W=W^*,\ \|W\|_1\le1
\right\}.
\]
Since every \(A_i\) is block diagonal, replacing \(W\) by its pinching onto the common block-diagonal algebra does not change any of the traces \(\operatorname{tr}(A_iW)\) and does not increase \(\|W\|_1\). Hence we may assume that \(W\) is block diagonal. Using the Jordan decomposition of $W$
\[
W=W_+-W_-,
\qquad
W_+,W_-\succeq0,
\]
we have  $\operatorname{tr}W_+ + \operatorname{tr}W_-=\|W\|_1\le1$, and thus \(W_+,W_-\in\Delta_{\le1}\). But this shows that  $K_S^\circ
\subseteq
A_S(\Delta_{\le1})-A_S(\Delta_{\le1}),$
and by \eqref{lemma36:eq2} we consequently get  for some constant $C'''>0$,
\begin{equation} \label{lemma36:eq3}
    N\left(
K_S^\circ,
C\rho\,B_\infty^S
\right)
\le
\exp(C'''s).
\end{equation}
Now, let $t:=C_1\sqrt{sD}.$ Using $(tK_S)^\circ=\frac{1}{t}K_S^\circ$ the preceding covering estimate \eqref{lemma36:eq3}  implies
\begin{equation}\label{lemma36:eq5}
    N\left(
(tK_S)^\circ,
\frac{1}{s}B_\infty^S
\right)
\le
\exp(C'''s).
\end{equation}
But by the covering-to-measure principle of Theorem~\ref{thm:covering-measure}, the covering number estimate of \eqref{lemma36:eq5} gives then the desired Gaussian measure lower bound for the discrepancy body $K_S$, namely
\begin{equation}\label{lemma36:eq4}
    \gamma_s(c_0tK_S)\ge \exp(-C_2s)
\end{equation}
where \(c_0,C_2>0\) are  absolute constants.

Finally, by using the Gaussian measure lower bound of \eqref{lemma36:eq4} and applying  the Gaussian partial-coloring theorem, Theorem~\ref{thm:partial-coloring}, to the symmetric convex body \(c_0tK_S\) and the point \(y\in(-1,1)^S\), we obtain an increment \(z\in C_3tK_S\) such that
$y+z\in[-1,1]^S$ and $|\{i\in S:|y_i+z_i|=1\} | \ge cs.$ Since \(z\in C_3tK_S\), so that by definition of \(K_S\) we get
\[
\left\|
\sum_{i\in S}z_iA_i
\right\|
\le
C_3t
\le
C\sqrt{sD},
\]
which proves our lemma.
\end{proof}

\paragraph{Proof of multiscale block partial coloring lemma.}

Finally, we are now able to prove our main partial coloring lemma.  
\begin{proof}[Proof of Lemma \ref{lem:stagewise-multiscale-block-pc}]
Fix an active set \(S\subseteq[n]\), write \(s:=|S|\), and let \(y\in(-1,1)^S\). Let
\[
\Delta
:=
\left\{
X=\bigoplus_{\alpha\in\mathcal B}X^{(\alpha)}:
X^{(\alpha)}\succeq0,\ 
\sum_{\alpha\in\mathcal B}\operatorname{tr}X^{(\alpha)}=1
\right\}
\]
denote the full block-diagonal spectraplex associated with the common block decomposition.

We first consider the case \(s\ge J\). By Lemma~\ref{lem:product-strata-net}, there exists a finite relative-entropy net $T_S\subseteq\Delta$ of the spectraplex $\Delta$ with $|T_S|\le \exp(Cs)$ such that for every \(X\in\Delta\) there exists \(Y\in T_S\) satisfying $S(X\|Y)
\le
C\log\left(e+\frac{Q}{s}\right).$ Thus by applying Lemma~\ref{lem:entropy-net-to-partial-coloring} with $D=C\log\left(e+\frac{Q}{s}\right)$, shows that there is a partial coloring   increment \(z\in\mathbb R^S\) with $y+z\in[-1,1]^S$  and $|\{i\in S:|y_i+z_i|=1\}|\ge cs$ satisfying,
\[
\left\|
\sum_{i\in S}z_iA_i
\right\|
\le
C\sqrt{s\log\left(e+\frac{Q}{s}\right)}.
\]
This proves the desired bound in the case \(s\ge J\).

Now suppose \(s<J\). In this regime we ignore the stratification and treat the full block decomposition as a single stratum. The total dimension is $M:=\sum_{\alpha\in\mathcal B}b_\alpha,$
and the largest block size is $b_{\max}:=\max_{\alpha\in\mathcal B}b_\alpha.$
Applying Lemma~\ref{lem:one-stratum-net} to this single stratum with parameter $\tau:=s$
gives a finite relative-entropy net \(T_S\subseteq\Delta\) with  $|T_S|\le \exp(Cs)$ such that  for every \(X\in\Delta\) there exists \(Y\in T_S\) satisfying $S(X\|Y)
\le
C\log\left(e+\frac{b_{\max}M}{s}\right).$
Applying Lemma~\ref{lem:entropy-net-to-partial-coloring} with $D=C\log\left(e+\frac{b_{\max}M}{s}\right)$
gives an partial coloring increment \(z\in\mathbb R^S\) with $y+z\in[-1,1]^S$ and $|\{i\in S:|y_i+z_i|=1\}|\ge cs$ satisfying
\[
\left\|
\sum_{i\in S}z_iA_i
\right\|
\le
C\sqrt{s\log\left(e+\frac{b_{\max}M}{s}\right)},
\]
which proves the desired bound in the case \(s<J\).
Finally, it remains to prove  upper bounds on \(b_{\max}\), \(M\). By assumption we have $Q=\sum_{j=1}^Jh_jm_j\le C_0n.$
Since each stratum is nonempty and \(h_j\) is the maximum block size in the stratum, we have $m_j\ge h_j$, so that also
$h_j^2\le h_jm_j\le Q$ for every \(j\).  Therefore $b_{\max}=\max_jh_j\le \sqrt Q\le \sqrt{C_0n}$.
Summing over \(j\) and using \(h_j\ge1\), we get $M=\sum_{j=1}^Jm_j
\le
\sum_{j=1}^Jh_jm_j
=
Q
\le C_0n.$
Consequently, we have that $b_{\max}M
\le
\sqrt{C_0n}\cdot C_0n
=
C_0^{3/2}n^{3/2}$, which completes our proof.
\end{proof}

\subsection{Proof of Theorem \ref{thm:block-diag-spencer}}

Finally, we will prove now the block diagonal Matrix Spencer Theorem by iterating Lemma~\ref{lem:stagewise-multiscale-block-pc}.
Write $Q=\sum_{j=1}^J h_jm_j$ and let $y^{(0)}=0\in[-1,1]^n$ and denote with $S_0:=[n]$ the initial active set. Suppose that after \(k\) stages we have constructed
$y^{(k)}\in[-1,1]^n$
and let $S_k:=\{i\in[n]: |y_i^{(k)}|<1\}$
be the current active coordinates. Let $s_k:=|S_k|$ denote the number of active coordinates and if \(s_k=0\), i.e there no active coordinates left and \(y^{(k)}\in\{\pm1\}^n\), then we stop with the iteration.
Suppose that \(s_k>0\) and apply Lemma~\ref{lem:stagewise-multiscale-block-pc} to the active set \(S_k\) and the point \(y^{(k)}|_{S_k}\in(-1,1)^{S_k}\) to obtain a partial coloring increment $z^{(k)}\in\mathbb R^{S_k}$ with $y^{(k)}|_{S_k}+z^{(k)}\in[-1,1]^{S_k}$ such that  at least \(c s_k\) active coordinates become integral, and have discrepancy
\begin{equation}\label{thm32:eq1}
    \left\|
\sum_{i\in S_k}z_i^{(k)}A_i
\right\|
\le
\Delta(s_k),
\end{equation}
where
\[
\Delta(s)
\le
C
\begin{cases}
\sqrt{s\log\left(e+\dfrac{Q}{s}\right)}, & s\ge J,\\[1.2em]
\sqrt{s\log\left(e+\dfrac{b_{\max}M}{s}\right)}, & s<J.
\end{cases}
\]
To continue the iteration, extend \(z^{(k)}\) by zero outside \(S_k\), and define $y^{(k+1)}:=y^{(k)}+z^{(k)}$. Then, we have that $y^{(k+1)}\in[-1,1]^n$
and the number of active coordinates decreases by a constant factor $s_{k+1}\le (1-c)s_k$.
Thus, after \(O(\log n)\) stages, all but at most \(O(1)\) coordinates have been colored. The remaining \(O(1)\) coordinates can be rounded arbitrarily, contributing only \(O(1)\) to the final discrepancy.

It remains to bound the total discrepancy accumulated over the partial-coloring stages. Since $y^{(0)}=0$,
and the final coloring \(x\in\{\pm1\}^n\) is obtained by summing the increments, the triangle inequality gives us by using \eqref{thm32:eq1}
\begin{equation}\label{thm32:eq2}
    \left\|
\sum_{i=1}^n x_iA_i
\right\|
\le
\sum_k
\left\|
\sum_{i\in S_k}z_i^{(k)}A_i
\right\|
+
O(1)
\le
\sum_k \Delta(s_k)+O(1).
\end{equation}
Finally, to obtain the desired discrepancy bound, we have to upper bound $\Delta(s_k)$ appearing in inequality \eqref{thm32:eq2}. To this end, consider first the stages for which \(s_k\ge J\). Since \(Q\le C_0n\), we have that $\Delta(s_k)
\le
C(C_0)\sqrt{s_k\log\left(e+\frac{n}{s_k}\right)}$. Using that the active-set sizes decrease geometrically, we can group the stages dyadically to obtain,
\begin{equation}\label{thm32:eq3}
    \sum_{s_k\ge J}\Delta(s_k)
\le
C(C_0)
\sum_{\ell\ge0}
\sqrt{2^{-\ell}n\log\left(e+2^\ell\right)}.
\end{equation}
 Now, using that $\sum_{\ell\ge0}2^{-\ell/2}\sqrt{1+\ell}<\infty$, we then get finally the bound $\sum_{s_k\ge J}\Delta(s_k)
\le
C(C_0)\sqrt n$, for some constant $C(C_0)>0$.  

We now consider the stages for which \(s_k<J\). Using Lemma~\ref{lem:stagewise-multiscale-block-pc}, we get that $b_{\max}M\le C(C_0)n^{3/2}$, and thus $\Delta(s_k)
\le
C(C_0)
\sqrt{s_k\log\left(e+\frac{n^{3/2}}{s_k}\right)}$. 
Again using similar ideas as in \eqref{thm32:eq3}, we obtain
\begin{equation}\label{thm32:eq4}
\sum_{s_k<J}\Delta(s_k)
\le
C(C_0)
\sqrt{
J\log\left(e+\frac{n^{3/2}}{J}\right)
}.
\end{equation}
Since \(J\le C_0\log(en)\), the bound in \eqref{thm32:eq4} is at most $C\log(en)\le C\sqrt n$ so that $\sum_{s_k<J}\Delta(s_k)
\le
C(C_0)\sqrt n$. 
Combining the two regimes and the \(O(1)\) final rounding contribution, we finally obtain for  \eqref{thm32:eq2} the discrepancy bound 
\[
\left\|
\sum_{i=1}^n x_iA_i
\right\|
\le
C(C_0)\sqrt n.
\]
This completes the proof of the block diagonal Matrix Spencer theorem.

\subsection{Proof of Algebraic Matrix Spencer Theorem}

In this section we will prove our main theorem.  The multiscale theorem from the previous section was stated in terms of a block stratification and the associated quantity $Q=\sum_{j=1}^J h_jm_j$. 
We now derive the form that will be used for finite-dimensional \(C^*\)-algebras. Suppose the matrices are simultaneously block diagonal with common block sizes \(b_\alpha\). If the square-sum $\sum_{\alpha\in\mathcal B} b_\alpha^2$ is \(O(n)\), then a dyadic stratification of the blocks has multiscale complexity \(Q=O(n)\). Thus, Theorem \ref{thm:block-diag-spencer} would give an \(O(\sqrt n)\) signing. Below we formulate a version of the Matrix Spencer theorem for block diagonal matrices in which the auxiliary stratification does not appear in the statement, rather formulating the discrepancy upper bound in terms of the intrinsic parameter $\sum_{\alpha\in\mathcal B} b_\alpha^2$ which will later become the parameter controlling the discrepancy in the $C^*$-algebra setting. 

\begin{theorem}[Square-sum block criterion]
\label{thm:square-summable-common-blocks}
Let \(A_1,\ldots,A_n\) be self-adjoint contractions which are block diagonal with respect to a fixed common block decomposition $H=\bigoplus_{\alpha\in\mathcal B}H_\alpha$ and denote $b_\alpha:=\dim H_\alpha$.
Assume that $\sum_{\alpha\in\mathcal B} b_\alpha^2\le C_1n$. 
Then there exist signs \(x_1,\ldots,x_n\in\{\pm1\}\) such that
\[
\left\|
\sum_{i=1}^n x_iA_i
\right\|
\le
O(\sqrt n)
\]
\end{theorem}

\begin{proof}
Our goal will be to construct a dyadic block stratification and then apply Theorem~\ref{thm:block-diag-spencer}. To this end, for \(k\ge0\), define
\[
\mathcal P_k
:=
\left\{
\alpha\in\mathcal B:
2^k\le b_\alpha<2^{k+1}
\right\},
\]
discarding the empty classes. These nonempty classes form a block stratification of \(\mathcal B\). For a nonempty class \(\mathcal P_k\), let $h_k:=\max_{\alpha\in\mathcal P_k} b_\alpha$ and $m_k:=\sum_{\alpha\in\mathcal P_k} b_\alpha$. Now,  since $h_k<2^{k+1}$  and \(b_\alpha\ge 2^k\) for any \(\alpha\in\mathcal P_k\), we have
\begin{equation}\label{thm37:eq1}
    h_km_k
=
h_k\sum_{\alpha\in\mathcal P_k}b_\alpha
\le
2^{k+1}\sum_{\alpha\in\mathcal P_k}b_\alpha
\le
2\sum_{\alpha\in\mathcal P_k}b_\alpha^2.
\end{equation}
Therefore, using \eqref{thm37:eq1} the multiscale complexity parameter $Q$ of this stratification satisfies $Q
=
\sum_k h_km_k
\le
2\sum_{\alpha\in\mathcal B}b_\alpha^2
\le
2C_1n$. It remains to check that the number of nonempty strata is logarithmic in \(n\). Indeed, since $\sum_{\alpha\in\mathcal B}b_\alpha^2\le C_1n$ 
every block size satisfies $b_\alpha\le \sqrt{C_1n}$, and hence for any $\cP_k\neq \emptyset$ we have that  $k\le \frac12\log_2(C_1n)$. Finally, we can apply 
 Theorem~\ref{thm:block-diag-spencer} to obtain a signing  $x\in\{\pm1\}^n$ such that 
\[
\left\|
\sum_{i=1}^n x_iA_i
\right\|
\le
C(C_1)\sqrt n,
\]
which completes the proof.
\end{proof}

We now turn to finite-dimensional \(C^*\)-algebras. By the Wedderburn decomposition, this setting reduces to the setting of Theorem~\ref{thm:square-summable-common-blocks}. The key point is that multiplicity spaces do not affect operator norms of signed sums, so the relevant parameter governing the discrepancy, the block-size square-sum is precisely \(\dim_{\mathbb C}\mathcal A\).

\begin{proof}[Proof of Algebraic Matrix Spencer Theorem \ref{thm:algebraic:spencer}]
We first prove the result under the additional assumption that the matrices
\(A_1,\ldots,A_n\) are self-adjoint. By the Wedderburn decomposition for finite-dimensional \(C^*\)-algebras, after conjugating by a unitary we may write
\[
\mathcal A
=
\bigoplus_{\alpha\in\mathcal B}
I_{r_\alpha}\otimes M_{d_\alpha}(\mathbb C)
\subseteq
\bigoplus_{\alpha\in\mathcal B}
B(\mathbb C^{r_\alpha}\otimes\mathbb C^{d_\alpha}).
\]
Thus each \(A_i\in\mathcal A\) has the form
\[
A_i
=
\bigoplus_{\alpha\in\mathcal B}
I_{r_\alpha}\otimes A_i^{(\alpha)},
\qquad
A_i^{(\alpha)}\in M_{d_\alpha}(\mathbb C).
\]
Since \(A_i\) is self-adjoint and \(\|A_i\|\le1\), also each block \(A_i^{(\alpha)}\) is self-adjoint and $\|A_i^{(\alpha)}\|\le1$ for all $i,\alpha$. 
Now, for every sign vector \(x\in\{\pm1\}^n\) we can write,
\[
\sum_{i=1}^n x_iA_i
=
\bigoplus_{\alpha\in\mathcal B}
I_{r_\alpha}\otimes
\left(
\sum_{i=1}^n x_iA_i^{(\alpha)}
\right),
\]
and hence for its spectral norm we get that,
\begin{equation} \label{thm13:eq1}
    \left\|
\sum_{i=1}^n x_iA_i
\right\|
=
\max_{\alpha\in\mathcal B}
\left\|
\sum_{i=1}^n x_iA_i^{(\alpha)}
\right\|.
\end{equation}
Therefore we may compress away the multiplicity spaces and consider instead the common-block matrices $\widetilde A_i
:=
\bigoplus_{\alpha\in\mathcal B} A_i^{(\alpha)}$
acting on $\widetilde H
:=
\bigoplus_{\alpha\in\mathcal B}\mathbb C^{d_\alpha}$.
These are self-adjoint contractions with common block sizes \(d_\alpha\). Now, using that $\sum_{\alpha\in\mathcal B}d_\alpha^2
=
\dim_{\mathbb C}\mathcal A
\le C_0n$,
we apply Theorem~\ref{thm:square-summable-common-blocks} to the family 
\(\widetilde A_1,\ldots,\widetilde A_n\), which shows that  there exist signs
\(x_1,\ldots,x_n\in\{\pm1\}\) such that
\begin{equation}\label{thm13:eq2}
    \left\|
\sum_{i=1}^n x_i\widetilde A_i
\right\|
\le
C(C_0)\sqrt n.
\end{equation}
But by the  identity in Equation \eqref{thm13:eq1} above and the discrepancy bound in \eqref{thm13:eq2} , the same signs satisfy
\[
\left\|
\sum_{i=1}^n x_iA_i
\right\|
\le
C(C_0)\sqrt n,
\]
which proves  the theorem for the self-adjoint case. 

The general case reduces to the self-adjoint case by the Hermitian dilation, i.e. by replacing the matrices $A_i$ by $H_i
:=
\begin{pmatrix}
0 & A_i \\
A_i^* & 0
\end{pmatrix}
\in M_2(\mathcal A)$. 
Then \(H_i\) is self-adjoint and $\|H_i\|=\|A_i\|\le1$. Moreover, the algebra \(M_2(\mathcal A)\) is finite-dimensional and satisfies $\dim_{\mathbb C}M_2(\mathcal A)=4\dim_{\mathbb C}\mathcal A\le 4C_0n$. Applying the proof of the self-adjoint case to \(H_1,\ldots,H_n\), we obtain signs
\(x_1,\ldots,x_n\in\{\pm1\}\) such that $\left\|
\sum_{i=1}^n x_iH_i
\right\|
\le
C(4C_0)\sqrt n$.  But since $\left\|
\sum_{i=1}^n x_iH_i
\right\|
=
\left\|
\sum_{i=1}^n x_iA_i
\right\|$, the same discrepancy bound also holds for the $A_i$.

\end{proof}

\section{Proof: Perturbations of Algebraic Families}
We recall the main partial coloring lemma of \cite{BansalJiangMeka2024MatrixSpencer}, adapted to our setting. 
\begin{lemma}[Main partial coloring of Bansal--Jiang--Meka ---  slightly adapted]
\label{lemma:bjm-package}
There are universal constants \(c,\eta,C>0\) with the following property. Let
\(L_1,\ldots,L_n\) be self-adjoint matrices satisfying
\[
\|L_i\|\le 1,
\qquad
\|L_i\|_F^2\le \frac{n}{\log^3(en)}
\]
for every \(i\in[n]\). Let \(S\subseteq[n]\), with \(s:=|S|\) be arbitrary and let $t_L(s):=
C\left(\sqrt{s}+n^{1/4}s^{1/4}\right)$.
Then there exist a subspace $H_S\subseteq \mathbb R^S$ with $\dim H_S\ge (1-\eta)s$ such that the  discrepancy body   $K_L(S):=\{z\in H_S:\left\|
\sum_{i\in S}z_iL_i
\right\|
\le
t_L(s)\}$  relative to this subspace satisfies
\[
\gamma_{H_S}\bigl(K_L(S)\bigr)\ge \exp(-\eta s).
\]
Consequently, for every point \(y\in(-1,1)^S\), there exists a partial coloring increment \(z\in\mathbb R^S\) such that $y+z\in[-1,1]^S$ and $|\{i\in S: |y_i+z_i|=1\}|\ge cs$ satisfying

\[
\left\|
\sum_{i\in S}z_iL_i
\right\|
\le
C\,t_L(s).
\]
\end{lemma}

We next show that the  partial-coloring body of a family of matrices generating a low dimensional algebra remains large in terms of Gaussian measure even after restricting to an arbitrary high-dimensional subspace. We will need such stability statement to  obtain a simultaneous partial coloring of the two matrix families: Lemma \ref{lemma:bjm-package} produces a subspace \(H_S\) on which the low-rank/Frobenius part is controlled, and we must verify that the partial-coloring body of a family of matrices generating a low dimensional algebra still has large Gaussian measure inside that same subspace.
The key point, that allows us to show such a statement, is that for matrices coming from a low dimensional algebra,   largeness of Gaussian volume of the discrepancy body is proven through polar covering estimates. These estimates are stable under orthogonal projection. Hence, if the algebraic discrepancy body is large in \(\mathbb R^S\), then its restriction to any subspace \(H\subseteq\mathbb R^S\) remains large with respect to the Gaussian measure on \(H\). This gives a mechanism which allows us to obtain a partial coloring for the two different matrix families simultaneously.

\begin{lemma}[Subspace stability of the algebraic discrepancy body]
\label{lem:algebraic-body-subspace-stability}
Let \(B_1,\ldots,B_n\) be self-adjoint contractions contained in a finite-dimensional \(C^*\)-algebra \(\mathcal A\) satisfying $\dim_{\mathbb C}\mathcal A\le C_0n$.
Let \(S\subseteq[n]\) be arbitrary with \(s:=|S|\), and  denote the discrepancy body $K_B(S)\subseteq\mathbb R^S$ by $K_B(S):=\{z\in \R^S:\left\|
\sum_{i\in S}z_iB_i
\right\|
\le t_B(s)\}$ where $t_B(s):=
C\sqrt{s\log\left(e+\frac{n}{s}\right)}$. Then, for any linear subspace $H\subseteq\R^S$ we have that
\begin{equation*}
    \gamma_H\bigl(K_B(S)\cap H\bigr)\ge \exp(-Cs),
\end{equation*}
where $C>0$ is some absolute constant.
\end{lemma}

\begin{proof}
Fix \(S\subseteq[n]\), write \(s:=|S|\), and let \(H\subseteq\mathbb R^S\) be a linear subspace.

By the Wedderburn, we may view the matrices \(B_i\) as common-block matrices with block sizes \(d_\alpha\) satisfying $\sum_\alpha d_\alpha^2
=
\dim_{\mathbb C}\mathcal A
\le C_0n.$
Thus we may consider again a dyadic stratification of these blocks as in the proof of Theorem \ref{thm:square-summable-common-blocks} with the multiscale parameter \(Q\le C(C_0)n\) and with \(J\le C(C_0)\log(en)\). Let $$K_0(S):=
\Bigl\{
z\in\mathbb R^S:
\left\|
\nlsum_{i\in S}z_iB_i
\right\|
\le 1
\Bigr\}$$ 
be the unit discrepancy body of the family $B_1,\dots,B_n$. By the entropy-net and polar-covering argument used in Lemmas~\ref{lem:product-strata-net}--\ref{lem:entropy-net-to-partial-coloring}, with the above dyadic stratification, we obtain
\begin{equation}\label{lem42:eq1}
    N\Bigl(
K_0(S)^\circ,
C\sqrt{\frac{D(s)}{s}}\,B_\infty^S
\Bigr)
\le
\exp(Cs),
\end{equation}
where $D(s):=C(C_0)\log\left(e+\frac{n}{s}\right)$. Since $K_B(S):=t_B(s)K_0(S)$ we have by  \eqref{lem42:eq1} that 
\begin{equation}\label{lem42:eq2}
    N\Bigl(
K_B(S)^\circ,
\frac{1}{s}B_\infty^S
\Bigr)
\le
\exp(Cs)
\end{equation}

It remains to show that \(K_B(S)\cap H\) has large Gaussian measure in \(H\). To this end, recall that for a symmetric convex set \(K\subseteq\mathbb R^S\), the polar of its restriction to \(H\) is $(K\cap H)^\circ_H = \operatorname{Proj}_H(K^\circ)$, where the polar on the left is taken inside the Euclidean space \(H\). Applying this to \(K=K_B(S)\), and projecting the preceding cover in Equation \eqref{lem42:eq2} gives,

\begin{equation}\label{lem42:eq3}
    N\left( (K_B(S)\cap H)^\circ_H, \operatorname{Proj}_H\left(\frac1sB_\infty^S\right) \right) \le \exp(Cs).
\end{equation}
Using  $ \frac1sB_\infty^S\subseteq \frac1{\sqrt s}B_2^S$ and that the orthogonal projection maps \(B_2^S\) into \(B_2^H\), it follows from Equation \eqref{lem42:eq3} that 
\begin{equation}\label{lem42:eq4}
    N\left(
(K_B(S)\cap H)^\circ_H,
\frac{1}{\sqrt{s}}B_2^H
\right)
\le
\exp(Cs).
\end{equation}

As \(\dim H\le s\), we have $ \frac1{\sqrt s}B_2^H \subseteq \frac1{\sqrt{\dim H}}B_2^H$, and therefore by \eqref{lem42:eq4} also 
\begin{equation}\label{lem42:eq5}
    N\left( (K_B(S)\cap H)^\circ_H, \frac1{\sqrt{\dim H}}B_2^H \right) \le \exp(Cs).
\end{equation}
Hence, using the covering in \eqref{lem42:eq5}, the covering-to-measure principle in Theorem \ref{thm:covering-measure} in the space \(H\) gives
 
\[
\gamma_H\bigl(K_B(S)\cap H\bigr)\ge \exp(-Cs).
\]
This completes the proof.
\end{proof}

Finally, we can prove Matrix Spencer for low-rank perturbations of algebraic
families.

\begin{proof}[Proof of Theorem \ref{thm:low-rank-plus-algebra}]
We prove the result again by a standard partial coloring iteration. To this end, let $y^{(0)}=0\in[-1,1]^n$ and set $S_0:=[n]$. Suppose that after \(k\) steps we have a fractional coloring \(y^{(k)}\in[-1,1]^n\) and let $S_k:=\{i\in[n]: |y_i^{(k)}|<1\}$ denote the active coordinates that still need to be colored. We denote with $s_k:=|S_k|$  number of active coordinates $S_k$ and we stop with the iteration if \(s_k=0\). Suppose \(s_k>0\). Apply Lemma~\ref{lemma:bjm-package} to the family \(\{L_i\}_{i\in S_k}\), which then shows that there exists a subspace $H_k\subseteq\mathbb R^{S_k}$ of dimension at least $\dim H_k\ge (1-\eta)s_k$ such that
\begin{equation}\label{thm15:eq1}
    \gamma_{H_k}\bigl(K_L(S_k)\bigr)\ge \exp(-\eta s_k),
\end{equation}
where  $K_L(S_k):=\{z\in H_{S_k}:\left\|
\sum_{i\in S_k}z_iL_i
\right\|
\le
t_L(s_k)\}$  and $t_L(s_k):=
C\left(\sqrt{s_k}+n^{1/4}s_k^{1/4}\right)$. 
Moreover,  applying Lemma~\ref{lem:algebraic-body-subspace-stability} to the algebraic part \(\{B_i\}_{i\in S_k}\) and to the subspace \(H_k\),  shows that 
\begin{equation}\label{thm15:eq2}
    \gamma_{H_k}\bigl(K_B(S_k)\cap H_k\bigr)\ge \exp(-C s_k),
\end{equation}
where $K_B(S_k):=\{z\in \R^S_k:\left\|
\sum_{i\in S_k}z_iB_i
\right\|
\le t_B(s_k)\}$ and $t_B(s_k):=
C\sqrt{s_k\log\left(e+\frac{n}{s_k}\right)}$.
Both \(K_L(S_k)\) and \(K_B(S_k)\cap H_k\) are symmetric convex subsets of \(H_k\). Hence, using \eqref{thm15:eq1} and \eqref{thm15:eq2} we get by the Gaussian correlation inequality (\cite{Royen2014GaussianCorrelation,LatalaMatlak2017Royen}),
\begin{equation}\label{thm15:eq3}
    \gamma_{H_k}\bigl(K_L(S_k)\cap K_B(S_k)\cap H_k\bigr)
\ge
\gamma_{H_k}\bigl(K_L(S_k)\bigr)
\gamma_{H_k}\bigl(K_B(S_k)\cap H_k\bigr)
\ge
\exp(-C' s_k).
\end{equation}
Therefore, we can apply Theorem \ref{thm:partial-coloring-relative}, the Gaussian partial-coloring theorem in the subspace \(H_k\) to the body $K_L(S_k)\cap K_B(S_k)\cap H_k$, and obtain a partial coloring increment $z^{(k)}\in \mathbb R^{S_k}$ to the point \(y^{(k)}|_{S_k}\)  such that 

\begin{equation}\label{thm15:eq4}
    \left\|
\sum_{i\in S_k}z_i^{(k)}L_i
\right\|
\le C\,t_L(s_k),
\qquad
\left\|
\sum_{i\in S_k}z_i^{(k)}B_i
\right\|
\le C\,t_B(s_k),
\end{equation}
where $y^{(k)}|_{S_k}+z^{(k)}\in[-1,1]^{S_k}$ and at least \(c s_k\) coordinates of \(y^{(k)}|_{S_k}+z^{(k)}\) lie in \(\{\pm1\}\).
To continue the iteration, extend \(z^{(k)}\) by zero outside \(S_k\), and define $y^{(k+1)}:=y^{(k)}+z^{(k)}$. Then the number of active coordinates decreases geometrically $s_{k+1}\le (1-c)s_k$.
Iterating this construction, and rounding the remaining \(O(1)\) coordinates arbitrarily at the end, we obtain a full coloring. Using \eqref{thm15:eq4} we get  by  triangle inequality that,
\begin{equation}\label{thm15:eq5}
    \Bigl\|
\nlsum_{i=1}^n x_iA_i
\Bigr\|
\le
C\sum_{k\ge0} t_L(s_k)
+
C\sum_{k\ge0} t_B(s_k)
+
O(1).
\end{equation}

In the following we will  bound the two terms  on the right-hand side of \eqref{thm15:eq5}. For the first term we obtain by using the geometric decay of \(s_k\), 
\begin{equation}\label{thm15:eq6}
  \sum_{k\ge0}\sqrt{s_k}
\le
\sqrt n\sum_{k\ge0}(1-c)^{k/2}
\le
C\sqrt n,  
\end{equation}
and also
\begin{equation}\label{thm15:eq7}
    \sum_{k\ge0}n^{1/4}s_k^{1/4}
\le
n^{1/4}n^{1/4}\sum_{k\ge0}(1-c)^{k/4}
\le
C\sqrt n.
\end{equation}
Therefore, combining both \eqref{thm15:eq6} and \eqref{thm15:eq7} we get, $\sum_{k\ge0}t_L(s_k)\le C\sqrt n$. 
To obtain a bound on the second term in \eqref{thm15:eq5} we  group the active-set sizes dyadically. The argument uses the same idea as in the proof of Theorem \ref{thm:block-diag-spencer}, we record it here simply for completeness.  For \(\ell\ge0\), let $I_\ell
:=
\left\{
k:
2^{-(\ell+1)}n<s_k\le 2^{-\ell}n
\right\}$.  Since \(s_{k+1}\le(1-c)s_k\), each dyadic band \(I_\ell\) contains at most \(O_c(1)\) indices. If \(k\in I_\ell\), then $s_k\le 2^{-\ell}n$ and $\frac ns_k\le 2^{\ell+1}$. 
Hence
\begin{equation}\label{thm15:eq8}
    t_B(s_k)
\le
C(C_0)\sqrt{2^{-\ell}n\log(e+2^{\ell+1})}
\le
C(C_0)\sqrt n\,2^{-\ell/2}\sqrt{1+\ell}.
\end{equation}
Using \eqref{thm15:eq8} and summing over the dyadic bands gives
\[
\sum_{k\ge0}t_B(s_k)
\le
C(C_0)\sqrt n
\sum_{\ell\ge0}2^{-\ell/2}\sqrt{1+\ell}
\le
C(C_0)\sqrt n.
\]

Thus, along the same partial-coloring iteration,
\[
\sum_{k\ge0}t_L(s_k)\le C\sqrt n,
\qquad
\sum_{k\ge0}t_B(s_k)\le C(C_0)\sqrt n,
\]
so that by \eqref{thm15:eq5} we finally obtain 
$\Bigl\|
\nlsum_{i=1}^n x_iA_i
\Bigr\|
\le
C(C_0)\sqrt n$.
\end{proof}

\section*{Statement on LLM Use}
Throughout this project, we used GPT-5.5 (Pro) extensively as an interactive proof assistant, both for expanding on our ideas, and for carrying out the technical details of proofs along paths that seemed most promising. The counterexample shared in the appendix was found by GPT, after we tried to coax it to extend the $O(\sqrt{n})$ claim for ``easy'' matrices to the variance formulation of the conjecture. GPT's role in helping prove the results was crucial toward speeding up our research on this work. We also used GPT's assistance with writing, but all final mathematical statements, proofs, and wording are by the authors, who take responsibility for the errors. 

\bibliographystyle{plainnat}
\bibliography{ref}
\appendix 
\section{A Counterexample to Variance-Sensitive Matrix Spencer}

The variance-sensitive strengthening considered here replaces the ambient scale $\sqrt n$ by the intrinsic variance parameter $\left\|\sum_{i=1}^n A_i^2\right\|_{\rm op}^{1/2}$. This strengthening is natural from the perspective of matrix concentration, where the same variance parameter governs the random spectral norm of $\sum_{i=1}^{n}x_{i}A_{i}$ for a random signing $x\in\{\pm\}^{n}$ up to a log factor. It is also consistent with the rank-one case for which Kyng, Luh, and Song proved the variance sensitive Matrix Spencer bound using techniques from \cite{marcus2015interlacing}.
Here we however show that this variance-sensitive strengthening cannot hold in full generality. 
The counterexample is  witnessed by a family of diagonal matrices, and this somewhat surprisingly  matches the predictions of the noncommutative Khintchine inequalities \cite{BandeiraBoedihardjoVanHandel2023}, that is, \emph{in the presence of commutativity the log factor cannot be removed}. Therefore, it might be  plausible that in the case when the family of matrices is highly noncommutative  the variance strengthening might hold.

\begin{theorem}\label{thm:diagctx}
There exists a family of diagonal matrices $A_{1}, \dots , A_{n} \in \mathbb R^{n\times n}$ with $\|A_i\|_{\op}\le 1$ and
\[
\text{disc}(A_{1},\dots, A_{n})  = \Theta(\sqrt{\log(n)})
\]
while 
\[
 \left\|\sum_{i=1}^n A_i^2\right\|_{\op}^{1/2}  = \Theta(1).
\]
\end{theorem}

\begin{proof}[Proof of Theorem \ref{thm:diagctx}]
Let $m \geq 2$ an integer, and set $n = 2^{m}$. We will index the matrix entries by elements $s = (s_{1}, \dots, s_{m})\in S$, and identify $A_{i} \in \mathbb R^{S\times S} $. 
Now, let $p= (1,\dots,1), q = (-1,\dots, -1)$ and choose a subset $U\subseteq S\setminus\{p,q\}$ with $|U|=m$ and write $F = U^{c}$ for the complement. Consider a bijection 
\[
\pi:\{m+1,\ldots,n\}\to F.
\]

We define our $n \times n$ diagonal matrices as follows. 
For \(i=1,\ldots,m\),  we set
\[
(A_i)_{s,s}=\frac{s_i}{\sqrt m},
\qquad s\in S,
\]
while for \(k=m+1,\ldots,n\), we let
\[
A_k=\frac1{\sqrt m}E_{\pi(k),\pi(k)},
\]
where we denote with \(E_{r,r}\)  the diagonal matrix with a single one in coordinate \(r\) and zero elsewhere.

We begin with computing the variance parameter. Since the matrices are diagonal, it suffices to compute diagonal entries of 
\[
\sum_{i=1}^n (A_i)_{s,s}^2 = \sum_{i=1}^m  \frac{s_i^2}{m} + \sum_{i=m+1}^n \frac{1}{m} \mathbf 1_{\{s = \pi(k))\}} = 1 + \frac{1}{m} \mathbf 1_{\{s\in F\}},
\]
and hence, $ \left\|\sum_{i=1}^n A_i^2\right\|_{\op} =  1 + \frac{1}{m}$.

Now, let $x\in\{\pm1\}^{n}$ be an arbitrary signing, and consider $X= \sum_{i=1}^n  x_i A_i$. For any $s\in S$ its diagonal entry is 
\[
X_{s,s}
=
\frac1{\sqrt m}\sum_{i=1}^m x_i s_i
+
\frac1{\sqrt m}
\mathbf 1_{\{s\in F\}}x_{\pi^{-1}(s)},
\]
and, in particular for $s=x$ we get then,
\[
|X_{x,x}|
=
|\sqrt{m}
+
\frac1{\sqrt m}
\mathbf 1_{\{x\in F\}}x_{\pi^{-1}(x)}| \ge
\sqrt m-\frac1{\sqrt m}
=
\frac{m-1}{\sqrt m}.
\]
Therefore we obtain that, 
\[
\left\|\sum_{i=1}^n x_i A_i\right\|_{\op} = \|X\|_{\op}
\ge
|X_{x,x}|
\ge
\frac{m-1}{\sqrt m}.
\]

Finally, it remains to show that there is in fact a signing that attains this lower bound. To this end, let 
$x_{1} = \cdots= x_{m} = 1$ and  $x_{k_p}=-1,  x_{k_q}=1$, where $k_p:=\pi^{-1}(p),  k_q:=\pi^{-1}(q)$, and let all the remaining coordinates be chosen arbitrarily. For any $s\in S$, we have that

\[
\sqrt m\, X_{s,s}
=
\sum_{j=1}^m s_j+\mathbf 1_{\{s\in F\}}x_{\pi^{-1}(s)}.
\]
In particular, for $s=p,q$ we get that

\[
\sqrt m\, X_{p,p}=m-1, \qquad \sqrt m\, X_{q,q}=-(m-1).
\]
Now, let $s\in S\setminus\{p,q\}$. Then, \(s\) is not the all-plus or all-minus vector, so that 
\[
|\sum_{j=1}^m s_j|\le m-2,
\]
and hence,
\[
\left|\sqrt m\, X_{s,s}\right|
\le |\sum_{j=1}^m s_j|+1
\le 
m-1.
\]
Therefore every scaled diagonal entry has absolute value at most \(m-1\), and the entries at
\(p\) and \(q\) attain this value, so that
\[
\left\|X\right\|_{\mathrm{op}}
=
\frac{m-1}{\sqrt m}.
\]
Together with the lower bound above, this then shows
\[
\inf_{x\in\{\pm1\}^n}
\left\|\sum_{i=1}^n x_{i} A_i\right\|_{\mathrm{op}}
=
\frac{m-1}{\sqrt m}.
\]

\end{proof}

\end{document}